\documentclass[12pt]{amsart}
\usepackage{pinlabel}
\usepackage{amsfonts,amsmath,amscd,xcolor,graphicx,amssymb}
\usepackage{caption}
\usepackage{epstopdf}
\usepackage{a4wide}

\theoremstyle{plain}
\newtheorem{theorem}{Theorem}[section]

\newtheorem{lemma}[theorem]{Lemma}

\theoremstyle{definition}

\newcommand{\N}{\operatorname{N}}

\theoremstyle{remark}

\makeatletter
\@namedef{subjclassname@2020}{\textup{2020} Mathematics Subject Classification}
\makeatother

\begin{document}

\title[Reducing spheres and weak reducing pairs]{Reducing spheres and weak reducing pairs for Heegaard surfaces in the $3$-sphere}

\author[S. Cho]{Sangbum Cho}
\thanks{The first-named author is supported by the National Research Foundation of Korea(NRF) grant funded by the Korea government(MSIT) (RS-2024-00456645).}
\address{Department of Mathematics Education, Hanyang University, Seoul 04763, Korea}
\email{scho@hanyang.ac.kr}

\author[Y. Koda]{Yuya Koda}
\thanks{The second-named author is supported by JSPS KAKENHI Grant Numbers JP23K20791, JP23H05437 and JP24K06744.}
\address{Department of Mathematics, Hiyoshi Campus, Keio University, Yokohama 223-8521, Japan, and
International Institute for Sustainability with Knotted Chiral Meta Matter (WPI-SKCM$^2$), Hiroshima University, Higashi-Hiroshima 739-8526, Japan}
\email{koda@keio.jp}

\author[J. H. Lee]{Jung Hoon Lee}
\thanks{The third-named author is supported by the National Research Foundation of Korea(NRF) grant funded by the Korea government(MSIT) (RS-2023-00275419).}
\address{Department of Mathematics and Institute of Pure and Applied Mathematics, Jeonbuk National University, Jeonju 54896, Korea}
\email{junghoon@jbnu.ac.kr}

%\subjclass[2020]{Primary: 57K30}
%\keywords{reducing sphere, primitive disk, the Powell Conjecture}

%\date{\today}

\begin{abstract}
Given a Heegaard surface in the $3$-sphere, we show that any non-separating weak reducing pair for the surface admits a reducing sphere that separates the two disks of the pair if and only if the genus of the surface is at most $3$.
\end{abstract}

\maketitle

%\begin{small}
%\hspace{2em}
%\textbf{2020 Mathematics Subject Classification}:
%57K30, 57K20; 20F05

%\hspace{2em}
%\textbf{Keywords}:
% Goeritz group, Heegaard splitting, the Powell Conjecture,

%\hspace{7.8em} reducing sphere, weak reducing disk
%\end{small}

\section{Introduction}\label{sec:introduction}

Every closed orientable $3$-manifold can be decomposed into two handlebodies of the same genus $g$ for some $g\geq 0$, which is called a genus-$g$ {\it Heegaard splitting} of the manifold.
The common boundary surface of the two handlebodies is called the {\it Heegaard surface} for the splitting.
In particular, from Waldhausen \cite{Wal68}, it is known that the $3$-sphere $S^3$ admits a unique genus-$g$ Heegaard splitting for each $g\geq 0$ up to isotopy.

Let $\Sigma$ be a genus-$g$ Heegaard surface in $S^3$, and let $V$ and $W$ be the two handlebodies bounded by $\Sigma$.
A pair of essential disks $D \subset V$ and $E \subset W$ is called a {\it weak reducing pair} for $\Sigma$ if $\partial D$ is disjoint from $\partial E$.
When referring to such a pair, we write $D - E$. 
If both $D$ and $E$ are non-separating, we call the pair a non-separating weak reducing pair. 
Of course $\Sigma$ admits a (non-separating) weak reducing pair only when $g \geq 2$, and for each $g \geq 2$ we have infinitely many such pairs.
A sphere $S$ in $S^3$ is called a {\it reducing sphere} for $\Sigma$ if $S \cap \Sigma$ is a single essential circle in $\Sigma$.
The two disks $S \cap V$ and $S \cap W$ are necessarily essential separating disks in $V$ and $W$, respectively.  
Again, $\Sigma$ admits a reducing sphere only when $g \geq 2$, and we have infinitely many such spheres in each case.

The notions of a reducing sphere and a weak reducing pair can be defined for a Heegaard surface $\Sigma$ of an arbitrary $3$-manifold in a similar manner.  
There is a close relationship between reducing spheres and weak reducing pairs.  
A reducing sphere for $\Sigma$ reduces $\Sigma$ to lower genus Heegaard surfaces, and by taking essential disks in the resulting handlebodies, we can easily find weak reducing pairs for $\Sigma$.  
Conversely, suppose that $\Sigma$ admits a weak reducing pair.
Then a maximal compression of $\Sigma$ either yields a reducing sphere for $\Sigma$, or results in an incompressible surface in the $3$-manifold. 
This is a seminal result in Heegaard theory due to Casson and Gordon \cite{C-G}.  
(The disks of the original weak reducing pair may or may not be the members of the collection of disks for the maximal compression.)

We say that a reducing sphere $S$ {\it separates} a weak reducing pair $D - E$ if $S$ is disjoint from $D \cup E$ and if $D$ and $E$ are contained in the different components cut off by $S$.
Given any non-separating weak reducing pair for a genus-$g$ Heegaard surface $\Sigma$ in $S^3$ with $g \geq 2$, it is natural to ask whether there always exists a reducing sphere that separates the pair. 
The answer looks affirmative at a glance, but it is not true in most cases.
The following is our main result.

\begin{theorem}\label{main_theorem}
Given a genus-$g$ Heegaard surface $\Sigma$ in $S^3$ with $g \geq 2$, any non-separating weak reducing pair for $\Sigma$ admits a reducing sphere that separates the pair only when $g$ is $2$ or $3$.
\end{theorem}

This work was motivated from the study of the Goeritz groups for $S^3$ by the authors.
For a genus-$g$ Heegaard splitting for $S^3$, the {\it Goeritz group} of the splitting is defined to be the group of isotopy classes of orientation-preserving self-homeomorphisms of $S^3$ that preserve each of the two handlebodies of the splitting setwise.
We denote by $\mathcal{G}_g$ simply the Goeritz group of the genus-$g$ Heegaard splitting of $S^3$.
It is easy to see that $\mathcal{G}_0 = 1$ and $\mathcal{G}_1 = \mathbb{Z} / 2 \mathbb{Z}$.
If $g \geq 2$, the group $\mathcal{G}_g$ is already infinite and much more complicated.
So far, it is known that $\mathcal{G}_2$ and $\mathcal{G}_3$ admit finite generating sets, but when $g \geq 4$, it is still an open question whether $\mathcal{G}_g$ is finitely generated or not.
For a brief history of the Goeritz groups for $S^3$, we refer the reader to the introduction of \cite{CKL24}.

One standard approach to obtain a (finite) generating set of a group is to construct a connected simplicial complex on which the group acts naturally and to study the action in detail.
For the Goeritz group $\mathcal{G}_g$, it was introduced the reducing sphere complex $\mathcal{R}(\Sigma_g)$ for a genus-$g$ Heegaard surface $\Sigma_g$ in $S^3$ for each $g \geq 2$ (see Scharlemann \cite{Sch04} and Zupan \cite{Zupan}).
The vertices of $\mathcal{R}(\Sigma_g)$ are the isotopy classes of reducing spheres for $\Sigma$, and a collection $\{v_0, v_1, \ldots, v_k\}$ of $k+1$ distinct vertices defines a $k$-simplex of $\mathcal{R}(\Sigma_g)$ if there exist representative spheres $S_0, S_1, \ldots, S_k$ of the vertices $v_0, v_1, \ldots, v_k$, respectively, satisfying that $|S_i \cap S_j \cap \Sigma_g| = 0$ if $g \geq 3$ and $|S_i \cap S_j \cap \Sigma_g| = 4$ if $g = 2$, for each $i, j \in \{0, 1, \ldots, k\}$ with $i \neq j$. 

It is shown in \cite{Sch04} and \cite{CKL24} that $\mathcal{R}(\Sigma_2)$ and $\mathcal{R}(\Sigma_3)$ are connected, respectively, which enables us to have finite generating sets for $\mathcal{G}_2$ and $\mathcal{G}_3$.
In the proof of the connectivity of $\mathcal{R}(\Sigma_3)$ in \cite{CKL24}, the property stated in Theorem~\ref{main_theorem} plays a key role. 
In fact, Lemma~4.1 in \cite{CKL24} presents a slightly stronger version of this property.
Furthermore, the connectivity of $\mathcal{R}(\Sigma_2)$ originally established in \cite{Sch04} can also be reproved using the same property.
Unfortunately, this property no longer holds for the Heegaard surface $\Sigma_g$ when $g \geq 4$, as stated in Theorem~\ref{main_theorem}. 
Nevertheless, we still conjecture that $\mathcal{R}(\Sigma_g)$ is connected for each $g \geq 4$.

\section{Proof of the main theorem}\label{sec:proof_of_main_theorem}

We break Theorem~\ref{main_theorem} into the following two lemmas.

\begin{lemma}\label{lem:genus_two_or_three}
Let $\Sigma$ be a Heegaard surface in $S^3$ with genus $2$ or $3$. 
Let $D - E$ be any non-separating weak reducing pair for $\Sigma$.
Then there exists a reducing sphere $P$ that separates the pair.
\end{lemma}

\begin{proof}
First, let $\Sigma$ be a genus-$2$ Heegaard surface, which splits $S^3$ into two genus-$2$ handlebodies $V$ and $W$.
We may assume that $D$ and $E$ are contained in $V$ and $W$, respectively.
Compressing $\Sigma$ along $D$ and $E$, we obtain a $2$-sphere $S$ on which the scars of $D$ and $E$ lie.
A simple closed curve $\gamma$ on $S$ splits $S$ into two disks, say $B_1$ and $B_2$.
We can choose such a curve $\gamma$ so that the two scars of $D$ lie in the interior of $B_1$ and the two scars of $E$ lie in the interior of $B_2$.
Then $\gamma$ bounds two essential separating disks in $V$ and $W$, respectively, that are disjoint from $D$ and $E$.
Then the union of the two separating disks are the desired reducing sphere $P$. 

\smallskip

Next, consider a genus-$3$ Heegaard surface $\Sigma$ in $S^3$, and let $V$ and $W$ be the genus-$3$ handlebodies cut off by $\Sigma$.
Again, we assume that $D$ and $E$ are contained in $V$ and $W$, respectively. 
Compressing $\Sigma$ along $D$ and $E$, we obtain a torus $\Sigma'$ which splits $S^3$ into two manifolds $V'$ and $W'$, where $E$ and $D$ are contained in the interiors of $V'$ and $W'$, respectively. 
Since $\Sigma'$ is a torus in $S^3$, at least one of $V'$ and $W'$ is a solid torus, and so we may assume that $V'$ is a solid torus.

On the other hand, let $V''$ be the genus-$2$ handlebody obtained by cutting $V$ along $D$. 
The manifold obtained by attaching a $2$-handle on $\partial V''$ along $E$ is then identified with the solid torus $V'$.
Thus, by Theorem~1 in \cite{Go}, there exists an essential disk $D'$ in $V''$ such that $\partial D'$ intersects $\partial E$ in a single point.
Consequently, we can find a non-separating essential disk $D''$ in $V''$ disjoint from $\partial E$ (and from $D'$).
Furthermore, we can choose such a disk $D''$ disjoint from the two scars of $D$ in $\partial V''$.
The disk $D''$ turns out to be a meridian disk of the solid torus $V'$ disjoint from all the scars of $D$ and $E$ in $\partial V'$.
Now, compressing the torus $\Sigma' = \partial V'$ along $D''$, we have a $2$-sphere $S$ on which the scars of $D$, $D''$ and $E$ lie.
We can choose a simple closed curve $\gamma$ in $S$ separating the scars of $E$ from the scars of $D$ and $D''$.
As in the case of $g = 2$, we obtain the desired reducing sphere $P$ such that $P \cap \Sigma = \gamma$.
\end{proof}

We simply denote by $\N(X)$ a regular neighborhood of $X$ and $\overline{X}$ the closure of $X$ for a subspace $X$ of a space $Y$. 
The exterior of $X$ is then $\overline{Y - \N(X)}$.
A knot or a connected spatial graph $K$ in $S^3$ is called {\it trivial} if the exterior of $K$ is a handlebody, otherwise $K$ is called {\it non-trivial}.
An arc $\tau$ in $S^3$ is called a {\it tunnel} for $K$ if $\tau$ meets $K$ only in its endpoints and the spatial graph $K \cup \tau$ is trivial.
If a non-trivial $K$ admits a tunnel, then $K$ is called {\it tunnel number one}.

\begin{lemma}\label{lem:genus_four_or_higher}
Let $\Sigma$ be a genus-$g$ Heegaard surface in the $3$-sphere with $g \geq 4$. 
Then there exists a non-separating weak reducing pair $D - E$ that does not admit a reducing sphere $P$ separating the pair.
\end{lemma}

\begin{proof}
We start with a well-known spatial graph in $S^3$, called {\it Suzuki's Brunnian graph} of order $n$, for each $n \geq 3$, denoted by $K_n$.
The graph $K_n$ is the union of $n$ simple arcs $\alpha_1, \alpha_2, \ldots, \alpha_n$ meeting only at their common endpoints as illustrated in Figure \ref{fig:graph} (a).
From \cite{Suz84}, $K_n$ is known to be non-trivial although every proper subgraph of $K_n$ is trivial.
Moreover, $K_n$ is tunnel number one, and one can verify easily that the arc $\tau$ in Figure \ref{fig:graph} (a) is a tunnel for $K_n$.
One more important property of $K_n$ is that its exterior admits a hyperbolic structure with totally geodesic boundary, which was shown in \cite{Ushi99}, and consequently, the exterior is boundary irreducible.
That is, there is no compressing disk for the exterior of $K_n$.

First, we choose the arcs $\alpha_1$ and $\alpha_2$, and then both of $V_0 = \N(\alpha_1 \cup \alpha_2)$ and $W_0 = \overline{S^3 - V_0}$ are solid tori.
(In fact, we can choose any other two arcs that do not contain the endpoints of the tunnel $\tau$. If $n = 3$, we can take $\tau$ whose endpoints lie in a single arc by sliding an endpoint of $\tau$.)
A regular neighborhood $\N(K_n)$ of $K_n$ is obtained by attaching $n-2$ $1$-handles $h_3, \ldots, h_n$ on $V_0$, where $h_3, \ldots, h_n$ correspond to the arcs $\alpha_3, \ldots, \alpha_n$, respectively.
Thus $\overline{W_0 - (h_3 \cup \cdots \cup h_n)}$ is the exterior of $K_n$.
Attaching a $1$-handle $h$ on $\N(K_n)$, where $h$ corresponds to the tunnel $\tau$, we have a genus-$n$ handlebody, and the closure of its complement is also a genus-$n$ handlebody, which is $\overline{W_0 - (h_3 \cup \cdots \cup h_n \cup h)}$.

\begin{center}
\labellist
\pinlabel {\large$K_n$} [B] at 90 203
\pinlabel {\small $\alpha_1$} [B] at 33 175
\pinlabel {\small $\alpha_2$} [B] at 61 175
\pinlabel {\small $\alpha_3$} [B] at 89 160
\pinlabel {\small $\alpha_{n-2}$} [B] at 147 160
\pinlabel {\small $\alpha_{n-1}$} [B] at 175 160
\pinlabel {\small $\alpha_n$} [B] at 200 160
\pinlabel {{\color{brown}$\tau$}} [B] at 215 94

\pinlabel {{\color{red}\large$V_0$}} [B] at 297 198
\pinlabel {\large$W_0$} [B] at 265 198
\pinlabel {\small $h_3$} [B] at 362 163
\pinlabel {\small $h_{n-2}$} [B] at 421 163
\pinlabel {\small $h_{n-1}$} [B] at 453 163
\pinlabel {\small $h_n$} [B] at 487 163
\pinlabel {{\color{brown}$h$}} [B] at 500 103

\pinlabel {(a)} [B] at 117 4
\pinlabel {(b)} [B] at 384 4
\endlabellist
\includegraphics[width=16.5cm]{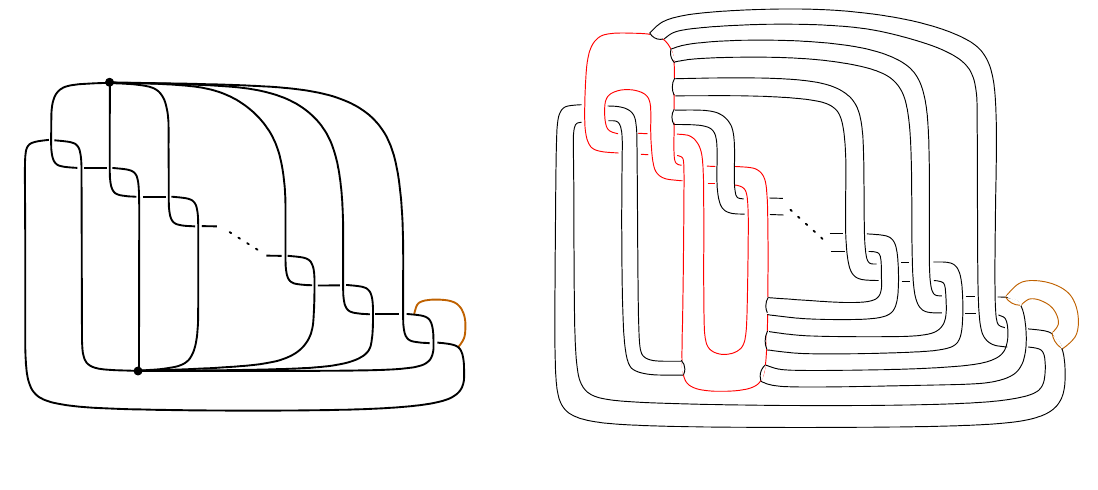}
\captionof{figure}{(a) Suzuki's Brunnian graph $K_n = \alpha_1 \cup \cdots \cup \alpha_n$ with its tunnel $\tau$. (b) The regular neighborhood of $K_n \cup \tau$.}
\label{fig:graph}
\end{center}

Next, we prepare any tunnel number one knot $K^*$ in $S^3$ and its tunnel $\tau^*$.
Then a regular neighborhood of $K^* \cup \tau^*$ can be obtained by attaching a $1$-handle $h^*$ on the solid torus $W^* = \N(K^*)$, where $h^*$ corresponds to the tunnel $\tau^*$.
The solid torus $W^*$ and the $1$-handle $h^*$ attached on $W^*$ in Figure \ref{fig:Heegaard surface} are obtained by a choice of a trefoil knot $K^*$ with its tunnel $\tau^*$.
Note that $W^* \cup h^*$ and $\overline{S^3 - (W^* \cup h^*)}$ are genus-$2$ handlebodies.
Now, by an embedding of $W_0$ into $S^3$, we identify $W_0$ with $W^*$.
For simplicity, we continue to denote the $1$-handles in $W^*$ inherited from $W_0$ by the same symbols $h_3, \ldots, h_n$ and $h$. 
We choose such an embedding so that the feet of $h^*$ is disjoint from those of $h_3, \ldots, h_n$ and $h$.

\begin{center}
\labellist
\pinlabel {\large$W_0$} [B] at 20 25

\pinlabel {\large$W^*$} [B] at 286 15
\pinlabel {{\color{blue}$h^*$}} [B] at 256 138
\pinlabel {{\color{blue}$E$}} [B] at 286 135
\pinlabel {{\color{brown}$h$}} [B] at 433 72
\pinlabel {{\color{brown}$D$}} [B] at 451 81
\endlabellist
\includegraphics[width=16cm]{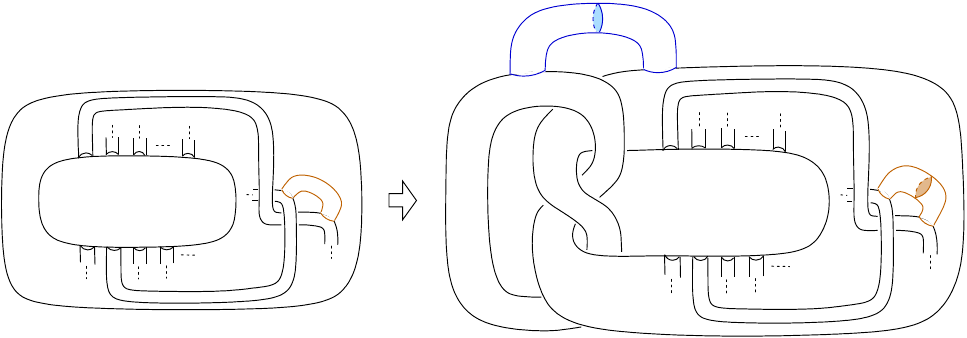}
\captionof{figure}{The non-separating weak reducing pair $D - E$ for the Heegaard surface $\Sigma$.}
\label{fig:Heegaard surface}
\end{center}

Let $V = \overline{S^3 - (W^* \cup h^*)} ~\cup h_3 \cup \cdots \cup h_n \cup h$ and $W = \overline{W^* - (h_3 \cup \cdots \cup h_n \cup h)}~\cup h^*$.
Then $V$ and $W$ are genus-$(n+1)$ handlebodies which form a Heegaard splitting of $S^3$.
Their common boundary surface is a genus-$g$ Heegaard surface in $S^3$, where $g = n+1 \geq 4$, which is our surface $\Sigma$.
We choose co-core disks $D$ and $E$ of the $1$-handles $h$ and $h^*$, respectively. 
They are then essential disjoint non-separating disks in $V$ and $W$, respectively, and so we have a non-separating weak reducing pair $D - E$ for $\Sigma$.

We will show that $D - E$ is the pair we desired.
Suppose, for contradiction, that there exists a reducing sphere $P$ for $\Sigma$ that separates the pair.
We may assume that $P$ is disjoint from the $1$-handles $h$ and $h^*$.
Compressing $\Sigma$ along $D$ and $E$ so that the two scars of $D$ are the feet of $h$ and the two scars of $E$ are the feet of $h^*$, we have a genus-$(g-2)$ surface, which we denote by $\Sigma_{D, E}$.
The separating circle $P \cap \Sigma$ is of course essential in $\Sigma$.
Furthermore, we have the following.

\medskip

\noindent{\it Claim.}
The separating circle $P \cap \Sigma_{D, E}$ is essential in $\Sigma_{D, E}$.

\smallskip

\noindent{\it Proof of Claim.}
Suppose not. 
Then the circle $P \cap \Sigma_{D, E}$ would bound a disk in $\Sigma_{D, E}$, and the disk contains only the two scars of $D$ or only the two scars of $E$ in its interior.
Consider the former case first.
Compressing $\Sigma$ along the separating disk $P \cap W$, we have two Heegaard surfaces whose genera are $1$ and $g-1$, respectively. 
The genus-$(g-1)$ Heegaard surface is isotopic to the boundary of $\overline{W^* - (h_3 \cup \cdots \cup h_n)} \cup h^*$.
We know that $\overline{W^* - (h_3 \cup \cdots \cup h_n)}$ is not a handlebody since it is homeomorphic to the exterior of $K_n$.
Thus $\overline{W^* - (h_3 \cup \cdots \cup h_n)} \cup h^*$ cannot be a handlebody, a contradiction.
For the latter case, compressing $\Sigma$ along the separating disk $P \cap V$, we have a genus-$(g-1)$ Heegaard surface in the same manner, which is isotopic to the boundary of $\overline{S^3 - W^*} \cup h_3 \cup \cdots \cup h_n \cup h$.
But we know that $\overline{S^3 - W^*}$ is not a handlebody since it is the exterior of the non-trivial knot $K^*$.
Thus $\overline{S^3 - W^*} \cup h_3 \cup \cdots \cup h_n \cup h$ cannot be a handlebody, a contradiction again.

\medskip

From the claim, we see that the circle $P \cap \Sigma_{D,E}$ bounds separating essential disks on both sides.  
However, $\Sigma_{D,E}$ is the boundary of $\overline{W^* - (h_3 \cup \cdots \cup h_n)}$, which is homeomorphic to the exterior of $K_n$.  
Since the latter is boundary irreducible, as noted in the first paragraph of the proof, this leads to a contradiction. 
\end{proof}

\bibliographystyle{amsplain}

\end{document}